\documentclass[10pt]{amsart}

\usepackage[latin1]{inputenc}
\usepackage{amsmath}
\usepackage{amsfonts}
\usepackage{amssymb}
\usepackage{ytableau}
\usepackage[mathscr]{eucal}
\usepackage{diagrams}
\usepackage{xy,color,graphicx}
\usepackage{hyperref}
\xyoption{all}
\usepackage{array}
\usepackage{enumerate}
\usepackage{url}
\usepackage{tikz}

\newcommand{\N}{\mathbb{N}}
\newcommand{\Q}{\mathbb{Q}}
\newcommand{\C}{\mathbb{C}}
\newcommand{\Z}{\mathbb{Z}}

\newcommand{\F}{\mathbb{F}}

\newtheorem{proposition}{Proposition}

\newtheorem{example}{Example}

\newtheorem{lemma}{Lemma}

\title{Three arithmetic sites}
\author{Lieven Le Bruyn} 
\address{Department of Mathematics, University of Antwerp \\ 
 Middelheimlaan 1, B-2020 Antwerp (Belgium) \\ {\tt lieven.lebruyn@uantwerpen.be}}

\begin{document}
\sloppy 

\maketitle

\begin{abstract}
Two new arithmetic sites are introduced, based on dynamical Belyi maps and Conway's big picture, respectively. We relate these to arboreal Galois representations, Bost-Connes data, and the original arithmetic site due to Connes and Consani.
\end{abstract}

\vskip 5mm

Renewed interest in the presheaf topos of a monoid category  (the one object category with given endomorphism monoid)  was spurred by the discovery of Connes and Consani, see \cite{CCAS}, that in the special case of $A=\N^{\times}_+$, the multiplicative monoid of strict positive numbers, the points of the so called {\em arithmetic site}, that is, of the corresponding presheaf topos
\[
\mathbf{points}(\widehat{\mathbf{A}}) = \Q^{\times}_+ \backslash \mathbb{A}^f / \widehat{\Z}^* \]
are the finite ad\`ele classes, a horrific topological space, better studied with the tools coming from non-commutative geometry.

In this note we will introduce two more such sites and investigate their interplay with the arithmetic site. The first one, {\em Belyi's site} $\widehat{\mathbf{B}}$, is constructed from the monoid of dynamical Belyi maps, as in \cite{Anderson}, and is relevant to the study of the absolute Galois group $Gal(\overline{\Q}/\Q)$, whereas the arithmetic site is relevant to the abelianisation $Gal(\overline{\Q}/\Q)^{ab}$. The second one, {\em Conway's site} $\widehat{\mathbf{C}}$, is constructed from a monoid structure on a subset of $\Gamma \backslash PGL_2^+(\Q)$, classifying projective classes of commensurable integral lattices of rank two, and is relevant to the description of groups appearing in moonshine, see for example \cite{Conway} or  \cite{LBMoonshine}.

We will study the points of these sites by constructing suitable localic covers of them. We will construct sheaves on these covers, relevant in the study of {\em arboreal Galois representations}, as in \cite{BouwDBM} (in the $\widehat{\mathbf{B}}$ case), and of {\em Bost-Connes data}, as introduced in \cite{MTabBC} (in the $\widehat{\mathbf{C}}$ case). Moreover, these three arithmetic sites are related via a triangle of geometric morphisms
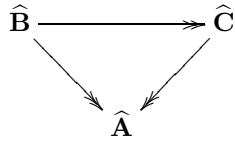
\[
\xymatrix{\widehat{\mathbf{B}} \ar@{->>}[rd] \ar@{->>}[rr] & & \widehat{\mathbf{C}} \ar@{->>}[ld] \\
& \widehat{\mathbf{A}} &} \]

\section{The arithmetic site and its localic cover}

In order to illustrate the strategy we will use later on in studying the Belyi and Conway sites, let us recall our approach to the original arithmetic site $\widehat{\mathbf{A}}$, see \cite{LBSieve} for more details. 

With $\mathbf{A}_L$ we denote the poset category of the poset structure on $\N_+^{\times}$ determined by reverse division, that is, $\mathbf{A}_L$ has an object $[k]$ for every $k \in \N_+^{\times}$ and a unique morphism
$\xymatrix{[k] \ar[r] & [l]}$ whenever $l|k$. The covariant functor $\pi : \mathbf{A}_L \rTo \mathbf{A}$ sending the morphism $\xymatrix{[k] \ar[r] & [l]}$ to the endomorphism $\tfrac{k}{l}$ induces a geometric morphism between the presheaf toposes 
\[
\pi~:~\widehat{\mathbf{A}_L} \rTo \widehat{\mathbf{A}} \]
and, as $\mathbf{A}_L$ is a poset-category, $\widehat{\mathbf{A}_L}$ is a {\em locale} which we call the {\em localic cover} of $\widehat{\mathbf{A}}$. In \cite{LBSieve} it is shown that the points of $\widehat{\mathbf{A}_L}$ are in one-to-one correspondence with subsets $S \subseteq \N_+^{\times}$ closed under division and taking least common multiples. The {\em supernatural numbers} $\mathbb{S}$ is the multiplicative semigroup consisting of formal products $s = \prod_{p \in \mathbb{P}} p^{s_p}$ where $p$ runs over all prime numbers $\mathbb{P}$ and $s_p \in \N \cup \{ \infty \}$, so we have a one-to-one correspondence
 \[
 \mathbf{points}(\widehat{\mathbf{A}_L}) = \mathbb{S} \qquad S \mapsto s=\prod_{p \in \mathbb{P}} p^{max(d~:~ p^d \in S)} \]
We equip $\mathbb{S}$ with the {\em localic topology} having as its sets of opens the subsets
\[ \mathbb{X}_l(\cup_i n_i \N^{\times}_+) = \{ s \in \mathbb{S}~:~\exists i~:~n_i|s \} \]
By \cite[Theorem 1]{LBcovers} we know that the locale $\widehat{\mathbf{A}_L}$ is equivalent to the topos of sheaves of sets $\mathbf{Sh}(\mathbb{S},loc)$ of the topological space $\mathbb{S}$, equipped with the localic topology.

The upshot being that we can associate to every $S \in \widehat{\mathbf{A}_L}$, that is to every contravariant functor $S~:~\mathbf{A}_L \rTo \mathbf{Sets}$ a sheaf $\mathcal{S} \in \mathbf{Sh}(\mathbb{S},loc)$ of which the stalk $
\mathcal{S}_s$ in an infinite supernatural number $s \in \mathbb{S}-\N^{\times}_+$ extends the functor as $S([n]) = \mathcal{S}_n$ for $n \in \N^{\times~}_+$. Here a couple of examples, see also \cite[Remark 2.10 and \S 4.5]{HemelaerAT}:

{\em Finite fields: } Fix a prime number $p$ and let $S$ be the functor
\[
S~:~\mathbf{A}_L \rTo \mathbf{Sets} \qquad \begin{cases} [n] \mapsto \F_{p^n} \\
\xymatrix{[n] \ar[r] & [m]} \mapsto \F_{p^m} \rInto \F_{p^n} \end{cases} \]
for the canonical embedding of fields obtained from taking invariants under the Frobenius. The stalk of the coresponding sheaf $\mathcal{S}$ at a supernatural number $s \in \mathbb{S} - \N^{\times}_+$ is then the (infinite) algebraic extension
\[
\mathcal{S}_s = \cup_{n | s} \F_{p^n} \]
It is well known that algebraic extensions of $\F_p$ correspond up to isomorphism to supernatural numbers in this way. 

{\em UHF-algebras: } Let $S$ be the functor
\[
S~:~\mathbf{A}_L \rTo \mathbf{Sets} \qquad \begin{cases} [n] \mapsto M_n(\C) \\
\xymatrix{[n] \ar[r] & [m]} \mapsto M_m(\C) \rInto M_n(\C) \end{cases} \]
UHF-algebras are closures of chains of matrix algebras
\[
M_{n_1}(\C) \rInto M_{n_2}(\C) \rInto \hdots \]
with $n_1|n_2| \hdots$, and are known to be classified up to isomorphism by supernatural numbers, which in our set-up corresponds to taking the stalk $\mathcal{S}_s$ of the corresponding sheaf $\mathcal{S} \in \mathbf{Sh}(\mathbb{S},loc)$.

The geometric morphism $\pi : \widehat{\mathbf{A}_L} \rTo \widehat{\mathbf{A}}$ induces a surjection on the level of points
\[
\mathbb{S} = \mathbf{points}(\widehat{\mathbf{A}_L}) \rOnto \mathbf{points}(\widehat{\mathbf{A}}) = [ \mathbb{S} ] \]
by modding out all finite supernaturals, that is, $[ \mathbb{S}]$ is the set of classes for the equivalence relation
\[
s \sim s' \quad \Leftrightarrow \quad \exists n,m \in \N^{\times~}_+~:~n.s = m.s' \]
and it is easy to verify that indeed $[ \mathbb{S} ]$ coincides with the set of finite ad\`ele classes $ \Q^{\times}_+ \backslash \mathbb{A}^f / \widehat{\Z}^*$. This equivalence relation implies that the points of the arithmetic site correspond to {\em stable isomorphism classes} for the objects classified by a functor $S \in \widehat{\mathbf{A}_L}$. For example, two algebraic extentions of $\F_p$ (resp. two UHF-algebras) $\mathcal{S}_s$ and $\mathcal{S}_{s'}$ are stably isomorphic if there exist $m,n \in \N^{\times}_+$ such that
\[
\mathcal{S}_s \otimes \F_{p^n} \simeq \mathcal{S}_{s'} \otimes \F_{p^m} \qquad (resp. \quad \mathcal{S}_s \otimes M_n(\C) \simeq \mathcal{S}_{s'} \otimes M_m(\C)~). \]
It is known that stable isomorphism classes of UHF-algebras are its Morita equivalence classes. As noncommutative spaces are Morita equivalence classes of $C^*$-algebras, we can view (the points of) the arithmetic site as the {\em moduli space} of the noncommutative spaces corresponding to UHF-algebras.

\section{Conway's site and its localic cover}

In \cite{Conway} John H. Conway introduced his {\em big picture} which is a graph with vertices the left coset classes $\Gamma \backslash PGL_2^+(\Q)$, where $PGL_2^+(\Q)$ is the group of all elements in $PGL_2(\Q)$ with strictly positive determinant, and $\Gamma$ is the modular group $PSL_2(\Z)$. We can identify the set $\Gamma \backslash PGL_2^+(\Q)$ with $\Q_+ \times \Q/\Z$ via the assignment
\[
\Q_+ \times \Q/\Z \rTo \Gamma \backslash PGL_2^+(\Q) \qquad X=(M,\frac{g}{h}) \mapsto \Gamma.\begin{bmatrix} M & \frac{g}{h} \\ 0 & 1 \end{bmatrix} = \Gamma.\alpha_X \]
For two elements $X,Y \in \Q_+ \times \Q/\Z$ let $d_{XY}$ be the smallest positive rational number such that $d_{XY}.\alpha_X \alpha_Y^{-1} \in GL_2(\Z)$. The {\em hyper-distance} between the two elements is then
\[
\delta(X,Y) = det(d_{XY}.\alpha_X \alpha_Y^{-1}) \in \N_+ \]
Conway showed in \cite[p. 329]{Conway}  that this definition is symmetric and that the $log$ of the hyper-distance is a proper distance function on $\Gamma \backslash PGL_2^+(\Q)$.

The poset-category $\mathbf{C}_L$ has as its objects $X \in \Gamma \backslash PGL_2^+(\Q)$ and a unique arrow $Y \rTo X$ if and only if $\delta(1,X) \leq \delta(1,Y)$ and $\delta(1,Y)=\delta(1,X).\delta(X,Y)$ where $1$ is the class of the identity matrix. {\em Conway's big picture} is the graph derived from $\mathbf{C}_L$ by only drawing an edge for each arrow $Y \rTo X$ such that $\delta(X,Y)$ is a prime number. An object $M=(M,0)$ for $M \in \N^{\times}_+$ is called a {\em number-class}.

Conway showed that each $X=(M,\frac{g}{h})$ has exactly $p+1$ neighbours in the big picture at hyper-distance $p$ for any prime number $p$. These are the classes
\[
X_k = (\frac{M}{p},\frac{g}{hp}+\frac{k}{p}~\text{mod}~1)~\quad \text{for $0 \leq k < p$}~\quad~\text{and}~\quad X_p=(pM,\frac{pg}{h}~\text{mod}~1) \]
It is convenient to consider these $p$-neighbours as the classes corresponding to the matrices obtained by multiplying $\alpha_{X}$ {\em on the left} with the matrices
\[
P_k = \begin{bmatrix} \frac{1}{p} & \frac{k}{p} \\ 0 & 1 \end{bmatrix}~\quad~\text{for $0 \leq k < p$, and}~\quad~P_p = \begin{bmatrix} p & 0 \\ 0 & 1 \end{bmatrix} \]
The classes at hyper-distance a $p$-power from $1$ form a $p+1$-valent tree in the big picture, and the big picture itself 'factorizes' as a product of these $p$-adic trees, see \cite[p. 332]{Conway} or \cite[Lemma 3]{LBMoonshine}. That is, the class (or object) $X$ at hyper-distance $p^k$ from $1$ corresponds to a unique product $P_{i_b}.\hdots.P_{i_2}.P_{i_1}P_p^a$ of the matrices $P_i$ with $0 \leq i  < p$ and $P_p$ introduced before with $a+b=k$. Here, $p^a$ is the maximal number-class on the unique path in the big picture from $X$ to $1$. The matrices $P_i$ with $0 \leq i < p$ generate a free monoid and satisfy the relation $P_p.P_i=id$. Matrices $P_i$ and $Q_j$ corresponding to different prime numbers $p$ and $q$ satisfy a {\em meta-commutation relation}.

\begin{lemma} \label{meta2} For distinct primes $p$ and $q$ and for all $0 \leq i < p$ and $0 \leq j < q$ there exist unique $0 \leq k < p$ and $0 \leq l < q$ such that
\[
P_i.Q_j = Q_l.P_k \quad \text{and further we have} \quad P_p.Q_j = Q_a.P_p~\text{when $a=pj~\text{mod}~q$} \]
\end{lemma}

\begin{proof}
\[
P_i.Q_j = \begin{bmatrix} \frac{1}{pq} & \frac{iq+j}{pq} \\ 0 & 1 \end{bmatrix} \]
and as $0 \leq iq+j < pq$ we have unique $0 \leq k < p$ and $0 \leq l < q$ such that $iq+j = lp+k$. \end{proof}

\begin{lemma} \label{lemma2} An object $X=(M,\frac{g}{h})$ in $\mathbf{C}_L$ at hyperdistance $N=p^k \times q^l \times \hdots \times r^s=\delta(1,X)$ for prime numbers $p < q < \hdots < r$ can be identified uniquely with a product
\[
P_{i_1}P_{i_2} \hdots P_{i_b} Q_{j_1} Q_{j_2} \hdots Q_{j_d} \hdots R_{z_1} R_{z_2} \hdots R_{z_v} P_p^aQ_q^c \hdots R_r^u = L_X.K_X \]
for unique $0 \leq i_a < p$, $0 \leq j_b < q, \hdots, 0 \leq z_u < r$ and with $a+b=k,c+d=l,\hdots,u+v=s$. 
\end{lemma}

\begin{proof} A path in the big picture from $1$ to $X=(M,\frac{g}{h})$ of minimal length can be viewed as a product (left multiplication) $X_l X_{l-1} \hdots X_2 X_1$ with each $X_i$ one of the matrices $P_a,Q_b,\hdots,R_u$. Now use the meta-commutation relations of the previous lemma.
\end{proof}

Here, $K_X = P_p^aQ_q^b\hdots R_r^u$ is the unique number-class at minimal hyper-distance from $X=(M,\frac{g}{h})$. For a {\em number-like class} $X=(M, \frac{g}{h})$, that is with $M \in \N^{\times}_+$, we have $K_X=Mh$ and $N=Mh^2$.

\begin{proposition} The points of $\widehat{\mathbf{C}_L}$ are in one-to-one correspondence with equivalence classes of finite or infinite sequences $X_1 \geq X_2 \geq \hdots$ of classes in $\Gamma \backslash PGL_2^+(\Q)$ under the equivalence relation
\[
X_1 \geq X_2 \geq \hdots \sim Y_1 \geq Y_2 \geq \hdots \Leftrightarrow \forall~i, \exists~j,k~:~X_i \geq Y_j,~Y_i \geq X_k \]
In particular, the finite points correspond to the objects in $\mathbf{C}_L$.
\end{proposition}

\begin{proof} The points of the presheaf topos of a poset-category are the upwards closed subsets of the poset satisfying the condition that it contains for any two of its elements an element smaller or equal to both. To any sequence $X_1 \geq X_2 \geq \hdots$ we associate the subset $\{ X \in \Gamma \backslash PGL_2^+(\Q)~|~\exists i~:~X \geq X_i \}$. The claims follow from this. 
\end{proof}

\vskip 3mm

In \cite{MTabBC} M. Marcolli and G. Tabuada developed a broad generalisation of the classical Bost-Connes system, where roots of unity are replaced by a {\em Bost-Connes datum} $(\Sigma,\sigma_n)$ consisting of
\begin{enumerate}
\item{an Abelian group $\Sigma$, possibly equipped with a continuous action of a Galois group $G$, and}
\item{a family of commuting $G$-equivariant group endomorphisms $\sigma_n~:~\Sigma \rTo \Sigma$ such that $\sigma_n \circ \sigma_m = \sigma_{n \times m}$.}
\end{enumerate}
An illustrative example is taking $\Sigma = \overline{\Q}^*$, the algebraic numbers with the canonical Galois action, and with $\sigma_n$ the power map $x \mapsto x^n$, see \cite{MTabBC} for more instances. In many interesting examples the Bost-Connes datum $(\Sigma,\sigma_n)$ will satisfy the following additional properties
\begin{enumerate}
\item[(3)]{the kernel of $\sigma_n$ is the cyclic group of order $n$, and} 
\item[(4)]{there is a family of set-maps $s_n~:~\Sigma \rTo \Sigma$ commuting among themselves and with the $\sigma_m$ such that their restriction to $\Sigma_n = Im(\sigma_n)$ is a set-theoretic section of the image map
\[
\xymatrix{0 \ar[r] & C_n \ar[r] & \Sigma \ar@/^2ex/[r]^{\sigma_n} & \Sigma_n \ar@/^2ex/[l]^{s_n} \ar[r] & 0} \]
and $\Sigma$ is the disjoint union of $C_n$-translates of the subset $s_n(\Sigma_n)$.}
\end{enumerate}
Let us denote the group operation on $\Sigma$ additively, and consider for each prime number $p$ maps on $\Sigma$ (identified as vectors via $x \mapsto \begin{bmatrix} x \\ 1 \end{bmatrix}$) by left-multiplication by the 'matrices' for $0 \leq i < p$
\[
\mathcal{P}_i = \begin{bmatrix} s_p & x_{i,p} \\ 0 & 1 \end{bmatrix} \qquad \text{and} \qquad \mathcal{P}_p = \begin{bmatrix} \sigma_p & 0 \\ 0 & 1 \end{bmatrix} \]
where $C_p = Ker(\sigma_p) = \{ 0 = x_{0,p},x_{1,p},\hdots,x_{p-1,p} \}$. That is, we have for all $x \in \Sigma$
\[
\mathcal{P}_i.x = s_p(x) + x_{i,p} \quad \text{and} \quad \mathcal{P}_p.x = \sigma_p(x) \]
As in \cite{MTabBC}, let $\rho_n : \Sigma_n \rTo P(\Sigma)$ be the map sending an element $x \in \Sigma_n$ to the set of its pre-images under $\sigma_n$, then we have $\rho_p(x) = \{ \mathcal{P}_0.x,\mathcal{P}_1.x,\hdots,\mathcal{P}_{p-1}.x \}$. If we want these operators to satisfy the same meta-commutation relations as the $P_i$ and $Q_j$ we have to impose
\begin{enumerate}
\item[(5)]{for any two prime numbers $p$ and $q$, for all $0 \leq i < p$ and $0 \leq j < q$ we have for the unique $0 \leq k <p$ and $0 \leq l < q$ such that $i.q+j=l.p+k$ that
\[
s_p(x_{j,q})+x_{i,p}=s_q(x_{k,p})+x_{i,q} \quad \text{and} \quad \sigma_p(x_{j,q})=x_{a,q} \]
for $a \equiv p.j~mod~q$.}
\end{enumerate}

\begin{proposition} Let $(\Sigma,\sigma_p,s_p)$ be a Bost-Connes datum satisfying conditions $(1)-(5)$, then we have a presheaf $S_{\Sigma} \in \widehat{\mathbf{C}_L}$, defined by assigning to every object $X$ with unique identifier
\[
P_{i_1}P_{i_2} \hdots P_{i_b} Q_{j_1} Q_{j_2} \hdots Q_{j_d} \hdots R_{z_1} R_{z_2} \hdots R_{z_v} P_p^aQ_q^c \hdots R_r^u = L_X.K_X \]
the set
\[
S_{\Sigma}(X) = \mathcal{P}_{i_1} \circ \hdots \circ \mathcal{P}_{i_b} \circ \mathcal{Q}_{j_1} \circ \hdots \circ \mathcal{Q}_{j_d} \circ \hdots \circ \mathcal{R}_{z_1} \circ \hdots \circ \mathcal{R}_{z_v}(\Sigma_{K_X}) \]
and this defines a sheaf $\mathcal{S}_{\Sigma}$ on $\mathbf{points}(\widehat{\mathbf{C}_L})$ equipped with the localic topology.
\end{proposition}

There is a covariant functor $\delta : \mathbf{C}_L \rTo \mathbf{A}_L$ sending $X$ to $\delta(1,X)$.

\begin{proposition} The fiber $\delta^{-1}(n)$ consists of $\psi(n) = n \prod_{p | n}(1 + \frac{1}{p})$ elements where $\psi$ is Dedekind's function, and there is a canonical bijection
\[
\delta^{-1}(n) \qquad \leftrightarrow \qquad \mathbb{P}^1_{\Z/n\Z} \]
\end{proposition}

\begin{proof} The stabilizer subgroup of $1$ is $\Gamma$ and any class $X$ at hyperdistance $n$ maps under $\Gamma$ to another at hyperdistance $n$. Further, the joint stabilizer of $1=(1,0)$ and $n=(n,0)$ is the mdular group $\Gamma_0(n)$. The canonical bijection then comes from
\[
\Gamma / \Gamma_0(n) \simeq GL_2(\Z) / \widetilde{\Gamma}_0(n) \simeq \mathbb{P}^1_{\Z/n\Z} \]
where the right-most bijection comes from considering the action of $GL_2(\Z)$ on $\mathbb{A}^2_{\Z/n\Z}$ and noticing that the stabilizer subgroup of the line through $(1,0)$ is $\widetilde{\Gamma}_0(n)$.
\end{proof}

The mental image of the geometric morphism $\delta : \widehat{\mathbf{C}_L} \rTo \widehat{\mathbf{A}_L}$ might therefore be that of the structural morphism
\[
\pi : \mathbb{P}^1_{\Z} \rOnto \mathbf{Spec}(\Z) \]
with the inclusion $\N^{\times}_+ \rInto C$ given by $n \mapsto (n,0)$ corresponding to the zero section.

\vskip 3mm

We define the {\em Conway monoid} $C$, having $p$ generators $P_0,\hdots,P_{p-1}$ for each prime $p$, with $\{ P_0,\hdots,P_{p-1} \}$ generating a free monoid of rank $p$. For different prime numbers $p$ and $q$ we have the meta-commutation relations, given by lemma~\ref{meta2} for all $0 \leq i < p$ and $0 \leq j < q$
\[
P_i \ast Q_j = Q_l \ast P_k, \quad \text{if} \quad  iq+j=lp+k \]
It follows from lemma~\ref{lemma2} that the elements of $C$ are in one-to-one correspondence with the subcategory $\mathbf{C}_L'$ of $\mathbf{C}_L$ consisting of the objects $X \in \mathbf{C}_L$ such that $1=(1,0)$ is the number-class of minimal distance from $X$, and that the poset-structure on $\mathbf{C}_L'$ is related to the monoid product via
\[
Y \leq X \quad \Leftrightarrow \quad \exists (!) Z~:~Y=Z \ast X \]
(uniqueness follows because $C$ is left cancellative).

Note that the multiplication is concatenation of minimal paths from $1$. That is, if $X$ has minimal path $L_X$ from $1$ and $Z$ miniimal path $L_Z$, then $Z \ast X$ is the object of $\mathbf{C}_L'$ one reaches after walking along path $L_Z$ from $X$. If however $X$ and $Z$ are objects in $\mathbf{C}_L$ with number classes $K_X$ at minimal distance from $X$ and $L_X$ a minimal path from $K_X$ to $X$, and $K_Z$ the number class at minimal distance of $Z$ and $L_Z$ a minimal path from $K_Z$ to $Z$, one might consider the product
\[
Z \ast X = L_Z.L_X.K_Z.K_X \]
but this product does not always produces an object at hyperdistance $\delta(1,X)\delta(1,Z)$ from $1$ as $L_X$ might start with a term $P_0$ for some prime number dividing $K_Z$. For example, the only consistent choice of $P_0 \ast P_p$ would be $1$, which is not at the expected hyperdistance $p^2$. I thank Jens Hemelaer for pointing out this problem.

\vskip 3mm

With $\mathbf{C}$ we denote the corresponding monoid-category, whose presheaf-topos $\widehat{\mathbf{C}}$ will be {\em Conway's site}.
Because the Conway monoid $C$ is left cancellative, we have a covariant functor $\pi : \mathbf{C}_L' \rTo \mathbf{C}$ which assigns to the unique morphism $Y \rTo X$ between two objects of $\mathbf{C}_L'$ such that $Y \leq X$ the unique element $Z \in C$ such that $Y = Z \ast X$. 

We will call the points of $\mathbf{C}_L'$ the {\em non-commutative supernatural numbers} and denote them with $\mathbb{S}_{nc}$. That is, $\mathbb{S}_{nc}$ is the set of equivalence classes of finite or infinite products $\hdots \ast A_2 \ast A_1$ in $C$ under the equivalence relation
\[
\hdots A_2 \ast A_1 \sim \hdots B_2 \ast B_1 \Leftrightarrow \forall i, \exists j,k,U,V~:~\begin{cases} B_j \ast \hdots \ast B_1 = U \ast A_i \ast \hdots \ast A_1 \\
A_k \ast \hdots \ast A_1 = V \ast B_i \ast \hdots \ast B_1 \end{cases} \]

This functor induces a surjection on the points of the corresponding presheaf toposes, giving us that
\[
\mathbf{points}(\widehat{\mathbf{C}}) = [ \mathbb{S}_{nc} ] \]
where the equivalence relations on noncommutative supernatural numbers is given by restricting to their tails, that is
\[
[ \hdots \ast A_2 \ast A_1] \sim [ \hdots \ast B_2 \ast B_1] \quad \Leftrightarrow \quad \exists i,j~:~[ \hdots \ast A_{i+1} \ast A_i] = [ \hdots \ast B_{j+1} \ast  B_j] \]
The hyper-distance $\delta(1,X)$ from $X$ to $1$ gives a morphism of monoids $\delta : C \rTo \N^{\times}_+$ and hence a geometric morphism to the arithmetic site $\widehat{\mathbf{C}} \rTo \widehat{\mathbf{A}}$. 

Connes and Consani equip the arithmetic site $\widehat{\mathbf{A}}$ with a structure sheaf $\Z_{max}=(\Z \cup \{ -\infty \},max,+)$ (the fundamental semi-ring in characteristic $1$, with addition $x \oplus y = max(x,y)$ and multiplication $x \otimes y = x+y$) in order to restrict the large group of auto-equivalences of the topos $\widehat{\mathbf{A}}$, coming from permutations among prime numbers in the monoid $\N^{\times}_+$.

The geometric morphism $\widehat{\mathbf{C}} \rTo \widehat{\mathbf{A}}$ can be seen as an alternative to this rigidification problem. The only monoid morphism of $\N^{\times~}_+$ which lifts to a monoid morphism of the Conway monoid $C$, compatible with the morphism $\delta$, is the identity map because a class at hyperdistance $p$ must be mapped to another one at hyperdistance $p$.

\section{Belyi's site and some localic subcovers}

A {\em (dynamical) Belyi map} is a rational map $B : \mathbb{P}^1_{\overline{\Q}} \rTo \mathbb{P}^1_{\overline{\Q}} $ defined over a number field with exactly three ramification points which we assume to be $\{ 0,1,\infty \}$ (and such that $B(0)=0, B(1)=1$ and $B(\infty)=\infty$), see for example \cite{Anderson} or \cite{Wood}. If $B' : \mathbb{P}^1_{\overline{\Q}}  \rTo \mathbb{P}^1_{\overline{\Q}} $ is another dynamical Belyi map, then so is their composition $B' \circ B$, that is, dynamical Belyi maps form a monoid. We will restrict here to the submonoid of dynamical Belyi {\em polynomials} and call it the {\em Belyi monoid} $B_{\overline{\Q}} $ or $B$ if we restrict to polynomials defined over $\Q$.

To a (dynamical) Belyi map $B$ we associate its {\em (framed) dessin (d'enfant)} which is a bicolored graph $D(B)$ drawn on $\mathbb{P}^1_{\C}=S^2$ with
\begin{itemize}
\item{as its {\em black vertices} the inverse images $B^{-1}(0)$, (with the vertex $0$ labeled),}
\item{as its {\em white vertices} the inverse images $B^{-1}(1)$, (with the vertex $1$ labeled),}
\item{as its edges the inverse images $B^{-1}(0,1)$ of the open interval $(0,1)$.}
\end{itemize}
If $B$ is a (dynamical) Belyi polynomial, then $D(B)$ is a (framed) {\em tree} dessin, that is, a bicolored planar tree.

Two bicolored planar trees $T$ and $T'$ are said to be {\em combinatorial equivalent} is there is a color-preserving tree isomorphism
\[
\gamma : V(T) \rTo V(T') \quad \text{and} \quad \gamma : E(T) \rTo E(T') \]
preserving the orientation around each vertex. That is, for any vertex $v$ of $T$ if $c_v=(e_1,\hdots,e_k)$ is the cyclic order of edges encountered when walking counter-clockwise around $v$, then the cyclic order of edges encountered when walking counter-clockwise around $\gamma(v)$ must be $c_{\gamma(v)}=(\gamma(e_1),\hdots,\gamma(e_k))$. A {\em framed tree dessin} is a bicolored planar tree such that a black vertex is labeled $0$ and a white vertex is labeled $1$. Two framed tree dessins $T$ and $T'$ are {\em isomorphic} if there is a combinatorial equivalence $\gamma : V(T) \rTo V(T')$ such that $\gamma(0)=0$ and $\gamma(1)=1$.

Let $Aut(\overline{\Q}) = \{ L = \alpha x + \beta : \alpha \not= 0 \}$ be the group of non-constant linear polynomials under composition. We say that two Belyi polynomials $B$ and $B'$ are {\em equivalent} if $B' = B \circ L$ for some $L \in Aut(\overline{\Q})$. G. Shabat proved that $B \leftrightarrow D(B)$ gives a one-to-one correspondence between equivalence classes of Belyi polynomials and combinatorial equivalence classes of tree dessins, see \cite[Thm. 2.2.9]{Lando}.

\begin{proposition} Every framed tree dessin $T$ is isomorphic to the framed tree dessin $D(B)$ of a unique (!) dynamical Belyi polynomial $B$.
\end{proposition}

\begin{proof} Forgetting the framing of $T$ for a moment, there is Belyi polynomial $P \in \overline{\Q}[x]$ such that $T$ and the dessin $D(P)$ are combinatorially equivalent by Shabat's result. Let $\alpha$ and $\beta$ be the vertices of $D(P)$ corresponding to the labeled vertices $0$ and $1$ of $T$ under the equivalence. Consider
\[
L = (\beta - \alpha)x+\alpha \in Aut(\overline{\Q}) \]
then $B=P \circ L$ is a Belyi polynomial such that $B(0)=P(\alpha)=0$ and $B(1)=P(\beta)=1$, whence $B$ is a dynamical Belyi polynomial with framed tree dessin isomorphic to $T$. To prove unicity, assume $B'$ is another dynamical Belyi polynomial such that $D(B')$ is isomorphic to $D(B)$. Then, by Shabat's result there exists an $L=\alpha x + \beta \in Aut(\overline{\Q})$ such that $B' = B \circ L$. But then, the combinatorial equivalence on vertices is given by
\[
L~:~V(D(B')) \rTo V(D(B)) \]
and in particular $L(0)=0$ and $L(1)=1$ whence $\beta=0$ and $\alpha=1$ whence $B' = B$.
\end{proof}

\begin{example} The bicolored planar tree
\[
\begin{tikzpicture}[scale=.7]
\tikzstyle{every node}=[draw,shape=circle,fill=white];
\node (a) at (0,0) {};
\node (b) at (2,0) {};
\tikzstyle{every node}=[draw,shape=circle,fill=black];
\node (c) at (1,0) {};
\node (d) at (3,0) {};
\draw (a) -- (c)
(b) -- (c)
(d) -- (b);
\end{tikzpicture} \]
can be realised as the tree dessin of the Belyi polynomial (even dynamical Belyi polynomial) $B=x^2(3-2x)$, as $D(B)$ is
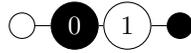
\[
\begin{tikzpicture}[scale=.7]
\tikzstyle{every node}=[draw,shape=circle,fill=white];
\node (a) at (0,0) {};
\node (b) at (2,0) {$1$};
\tikzstyle{every node}=[draw,shape=circle,fill=black];
\node (c) at (1,0) {$\color{white}{0}$};
\node (d) at (3,0) {};
\draw (a) -- (c)
(b) -- (c)
(d) -- (b);
\end{tikzpicture}
\]
with leftmost white vertex $-\frac{1}{2}$ and rightmost black vertex $\frac{3}{2}$. There are $3$ more $(b,w)$-couples: $a=(0,-\frac{1}{2}), b=(\frac{3}{2},1)$ and $c=(\frac{3}{2},-\frac{1}{2})$. The corresponding linear maps are $L_a=-\frac{1}{2}x$, $L_b=-\frac{1}{2}x+\frac{3}{2}$ and $L_c=-2x+\frac{3}{2}$, resulting in the dynamical Belyi polynomials with corresponding framed tree dessins
\[
\begin{tikzpicture}[scale=.7]
\tikzstyle{every node}=[draw,shape=circle,fill=black];
\node (a) at (0,0) {};
\node (b) at (2,0) {$\color{white}{0}$};
\tikzstyle{every node}=[draw,shape=circle,fill=white];
\node (c) at (1,0) {};
\node (d) at (3,0) {$1$};
\draw (a) -- (c)
(b) -- (c)
(d) -- (b);
\end{tikzpicture}  \qquad \leftrightarrow \qquad \frac{1}{4} x^2 (x+3) \]

\vskip 2mm

\[
\begin{tikzpicture}[scale=.7]
\tikzstyle{every node}=[draw,shape=circle,fill=black];
\node (a) at (0,0) {$\color{white}{0}$};
\node (b) at (2,0) {};
\tikzstyle{every node}=[draw,shape=circle,fill=white];
\node (c) at (1,0) {$1$};
\node (d) at (3,0) {};
\draw (a) -- (c)
(b) -- (c)
(d) -- (b);
\end{tikzpicture}  \qquad \leftrightarrow \qquad \frac{1}{4}x(x-3)^2 \]

\vskip 2mm

\[
\begin{tikzpicture}[scale=.7]
\tikzstyle{every node}=[draw,shape=circle,fill=black];
\node (a) at (0,0) {$\color{white}{0}$};
\node (b) at (2,0) {};
\tikzstyle{every node}=[draw,shape=circle,fill=white];
\node (c) at (1,0) {};
\node (d) at (3,0) {$1$};
\draw (a) -- (c)
(b) -- (c)
(d) -- (b);
\end{tikzpicture} \qquad \leftrightarrow \qquad x(4x-3)^2 \]
Note that the bicolored planar trees
\[
\begin{tikzpicture}[scale=.7]
\tikzstyle{every node}=[draw,shape=circle,fill=white];
\node (a) at (0,0) {};
\node (b) at (2,0) {};
\tikzstyle{every node}=[draw,shape=circle,fill=black];
\node (c) at (1,0) {};
\node (d) at (3,0) {};
\draw (a) -- (c)
(b) -- (c)
(d) -- (b);
\end{tikzpicture}  \qquad \text{and} \qquad
\begin{tikzpicture}[scale=.7]
\tikzstyle{every node}=[draw,shape=circle,fill=black];
\node (a) at (0,0) {};
\node (b) at (2,0) {};
\tikzstyle{every node}=[draw,shape=circle,fill=white];
\node (c) at (1,0) {};
\node (d) at (3,0) {};
\draw (a) -- (c)
(b) -- (c)
(d) -- (b);
\end{tikzpicture}  \]
are combinatorial equivalent as they differ by a rotation.
\end{example}

Following \cite[\S 6]{Shabat} we will identify the following parts of a framed tree dessin $T$:
\begin{enumerate}
\item{The {\em spine} $s(T)$ of $T$ is the unique path in $T$ from $0$ to $1$, }
\item{The {\em body} $b(T)$ of $T$ is the subtree consisting of the spine and all branches attached to a vertex in the spine, different from $0$ and $1$,}
\item{The {\em head} $h(T)$ of $T$ consists of all branches attached to $0$, except for the spine ,}
\item{The {\em tail} $t(T)$ of $T$ consists of all branches attached to $1$, except for the spine.}
\end{enumerate}
Given framed tree dessins $T$ and $T'$ one constructs a new framed tree dessin $T \circ T'$ as follows (see \cite[\S 6]{Shabat}):
\begin{enumerate}
\item{To each black vertex $b$ in $T'$ attach a number of heads of $T$ equal to the valency of $b$, and separate these copies by the edges in $T'$ at $b$ .}
\item{To each white vertex $w$ in $T'$ attach a number of tails of $T$ equal to the valency of $w$, and separate these copies by the edges in $T'$ at $w$.}
\item{Every edge $e$ in $T'$ will be oriented from the black vertex to the white vertex. Replace $e$ by a copy of the body of $T$ respecting the orientation, that is, such that the $0$-vertex in $T$ glues to the black vertex and the $1$-vertex to the white vertex of $e$.}
\item{The framing of $T \circ T'$ is inherited from that of $T'$.}
\end{enumerate}

If $T\simeq D(B)$ and $T' \simeq D(B')$, then $T \circ T' \simeq D(B \circ B')$, which allows us to study the Belyi monoids $B_{\overline{\Q}}$ and $B$ purely combinatorially by investigating the monoid structure on framed tree dessins. The absolute Galois group $Gal(\overline{\Q}/\Q)$ acts on (dynamical) Belyi polynomials and hence on (framed) tree dessins. We recall some Galois invariants and how they behave under composition.

For a (framed) tree dessin $T$ the {\em passport} $p(T)$ is the couple $(p_w,p_b)$ of partitions of $d$ which are the valencies $p_w$ (resp. $p_b$) of the white (resp. black) vertices of $T$. The {\em automorphism group} $Aut(T)$ is the group of combinatorial self-equivalences of $T$. This group is always cyclic and can be viewed as rotations around a central vertex preserving colours and valencies. The {\em monodromy group} $\mu(T)$ is the (conjugacy class) of the subgroup of $S_d$ generated by the permutation $\alpha$ having as its cycles the cyclic ordering of edges $c_v$ for all black vertices $v$, and $\beta$ with cycles $c_v$ for all white vertices of $T$.

If the passports of the head (minus $0$) of $T$ be $(h_b,h_w)$, of the tail (minus $1$) of $T$ be $(t_b,t_w)$, of the body (minus $0$ and $1$) of $T$ be $(b_b,b_w)$. Let the valency at $0$ in $T$ be $i$ and the valency at $1$ in $T$ be $j$.
Then, if the passport of $T'$ is $(p_b,p_w)$, and the degree of $B'$ is $d$, the passport of $T \circ T'$ is
\[
p(T \circ T') = (h_b^d \times b_b^d \times t_b^d \times (i.p_b),h_w^d \times b_w^d \times t_w^d \times (j.p_w)) \]

If $(\alpha_T,\beta_T)$ is the pair of permutations of $E=E(T)$ giving the cyclic order of edges around the back (resp. white) vertices of $T$, and let $(\alpha_{T'},\beta_{T'})$ be the permutations on the edges $F=E(T')$ of $T'$. Let $e_0 \in E$ be the edge of the spine of $T$ connected to $0$ and $e_1$ the edge of the spine of $T$ connected to $1$. 
If we identify $E \times F$ to the edges $E(T \circ T')$, then the permutations giving the cyclic order of edges of $T \circ T'$ around the black (resp. white) vertices are
\[
\alpha_{T \circ T'}(e,f) = \begin{cases} (\alpha_T(e),f)~\text{if $e \not= e_0$,} \\
(\alpha_T(e),\alpha_{T'}(f))~\text{if $e=e_0$}
\end{cases} \]
\[
\beta_{T \circ T'}(e,f) = \begin{cases} (\beta_T(e),f)~\text{if $e \not= e_1$,} \\
(\beta_T(e),\beta_{T'}(f))~\text{if $e=e_1$}
\end{cases}
\]

\begin{proposition} \label{cancel} The Belyi monoids $B_{\overline{\Q}}$ and $B$ are both left- and right-cancellative, that is, for all dynamical Belyi polynomials $B,B',B"$ 
\[
B \circ B' = B \circ B" \Rightarrow B' = B" \quad \text{and} \quad B' \circ B = B" \circ B \Rightarrow B'=B" \]
In fact, we have the stronger result that if $deg(B)=deg(B_1)$ and if
\[
B' \circ B = B" \circ B_1 \Rightarrow B'=B"~\text{and}~B=B_1 \]
\end{proposition}

\begin{proof} If $deg(B)=deg(B_1)$ and $B' \circ B = B" \circ B_1$, then by \cite[Lemma 2.4.18]{Lando} there exists $L = \alpha x + \beta \in Aut(\overline{\Q})$ such that $B_1 = L \circ B = \alpha B + \beta$. But then, from $B(0)=0=B_1(0)$ and $B(1)=1=B_1(1)$ it follows that $\alpha = 1$ and $\beta=0$ and therefore $B=B_1$. For every $c \in \C$ take a $d \in B^{-1}(c)$, then if $B' \circ B = B" \circ B$ we have that $B'(c) = B' \circ B(d) = B" \circ B(d) = B"(c)$, whence $B'=B"$.
\end{proof}

In \cite[Prop. 3.1]{Anderson} the following family of dynamical Belyi polynomials defined over $\Q$ was introduced: for each degree $d$ and  all $0 \leq k < d$ consider the polynomial
\[
B_{d,k} = c x^{d-k}(a_0 x^k + \hdots + a_{k-1} x + a_k) \in \Q[x] \]
with
\[
a_i = \frac{(-1)^{k-i}}{d-i} \binom{k}{i} \quad \text{and} \quad c = \frac{1}{k!} \prod_{j=0}^k (d-j) \]
Note that $B_{d,0} = x^d$ and $B_{d,d-1} = 1 - (1-x)^d$. 
The corresponding framed tree dessin $D(B_{d,k})=E_{d,k}$ has vertex $0$ of valency $d-k$, vertex $1$ of valency $k+1$ and all remaining vertices are leaf vertices. For example
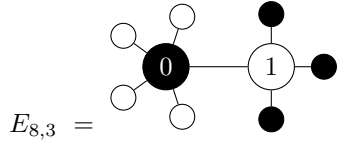
\[
E_{8,3}~=~\begin{tikzpicture}[scale=.7]
\tikzstyle{every node}=[draw,shape=circle,fill=black];
\path (0:0cm) node (v0) {$\color{white}{0}$};
\node (v7) at (3,0) {};
\node (v8) at (2,1) {};
\node (v9) at (2,-1) {};
\tikzstyle{every node}=[draw,shape=circle,fill=white];
\path (0:2cm) node (v1) {$1$};
\path (72:1cm) node (v2) {};
\path (2*72:1cm) node (v3) {};
\path (3*72:1cm) node (v4) {};
\path (4*72:1cm) node (v5) {};

\draw (v0) -- (v1)
(v0) -- (v2)
(v0) -- (v3)
(v0) -- (v4)
(v0) -- (v5)
(v1) -- (v7)
(v1) -- (v8)
(v1) -- (v9);
\end{tikzpicture}\]
$E_{d,k}$ has as passport of its head is $(1^{d-k-1},0)$, of the body $(0,0)$ and of the tail $(0,1^{k})$. But then, the passport of $E_{d,k} \circ T$ for a framed tree dessin $T$ with passport $(p_b,p_w)$ and $n$ edges
\[
p(E_{d,k} \circ T) = (1^{nk} \times (d-k).p_b,1^{n(d-k-1)} \times (k+1).p_w) \]

\begin{proposition} \label{free} The Belyi monoids $B_{\overline{\Q}}$ and $B$ contain free sub-monoids:
\begin{enumerate}
\item{If $\{ B_1,\hdots,B_k \}$ is a set of dynamical Belyi polynomials of fixed degree $d \geq 2$, then they generate a free sub-monoid of $B_{\overline{\Q}}$.}
\item{The dynamical Belyi polynomials $\{ B_{d,k}~|~1 \leq k \leq d-2, d \geq 2 \}$ generate a free sub-monoid of $B$.}
\end{enumerate}
\end{proposition}

\begin{proof} $(1)$ follows by induction on the length of words from proposition~\ref{cancel}. For $(2)$ suppose that $B$ and $B'$ are dynamical Belyi polynomials of degrees $n$ and $n'$ such that $B_{d,k} \circ B = B_{d',k'} \circ B'$ for $1 \leq k \leq d-2$ and $1 \leq k' \leq d'-2$. From the passport of $D(B_{d,k} \circ B)$ it follows that $D(B_{d,k} \circ B)$ has $n \times k$ black leaf-vertices and $n \times (d-k-1)$ white leaf vertices. This gives us the equalities
\[
n \times k = n' \times k',~n \times (d-k-1)=n' \times (d'-k'-1),~\text{and}~n \times d = n' \times d' \]
from which $n=n'$ follows, but then proposition~\ref{cancel} finishes the proof.
\end{proof}

Note that the Belyi monoids have an involution sending $B$ to $(1-x) \circ B \circ (1-x)$. At the level of the framed tree dessins the involution reverses the colours of the vertices, swaps the $0$ and $1$ vertices and rotates the dessin by 180 degrees.

\vskip 3mm

Let $\mathbf{B}_L$ (resp. $\mathbf{B}_{\overline{\Q}L}$) be the poset-category with objects the elements of the Belyi monoid $B$ (resp. $B_{\overline{\Q}}$) and a unique arrow $B \rTo B'$ if and only if $B \leq B'$, that is $B = B' \circ B"$ for some dynamical Belyi polynomial $B"$.

\begin{proposition} For a dynamical Belyi polynomial $B \in \overline{\Q}[x]$ let $\overline{\Q}(B-t)$ be the finite Galois extension of $\overline{\Q}(t)$ generated by the roots of the polynomial $B-t \in \overline{\Q}(t)[x]$, then
\[
S_{geo}~:~\mathbf{B}_{\overline{\Q}L} \rTo \mathbf{Sets} \qquad B \mapsto \overline{\Q}(B-t) \]
is a presheaf, that is a contravariant functor. For a dynamical Belyi polynomial $B \in \Q[x]$ let $\Q(B-t)$ be the finite Galois extension of $\Q(t)$ generated by the roots of the polynomial $B-t \in \Q(t)[x]$, then
\[
S_{arith}~:~\mathbf{B}_L \rTo \mathbf{Sets} \qquad B \mapsto \Q(B-t) \]
is a presheaf.
\end{proposition}

\begin{proof} For $B \rightarrow B'$ with $B = B' \circ B"$ let $\alpha$ be a root of $B'-t$. Then there is a root $\beta$ of $B-t$ such that $B"(\beta) = \alpha$. But then
\[
\overline{\Q}(t)(\alpha) \subseteq \overline{\Q}(t)(\beta) \subseteq \overline{\Q}(B-t) \]
As we can repeat this for all roots of $B'-t$, we have that $\overline{\Q}(B'-t)$ is a subfield of $\overline{\Q}(B-t)$. For $B=B' \circ B"$ in $\mathbf{B}_{L}$ we can repeat this argument giving us that
\[
\Q(t)(\alpha) \subseteq \Q(t)(\beta) \subseteq \Q(B-t) \]
and therefore $\Q(B'-t)$ is a subfield of $\Q(B-t)$.
\end{proof}

Recall that the monodromy group of the framed dessin $D(B)$ is isomorphic to the Galois group $Gal(\overline{\Q}(B-t)/\overline{\Q}(t))$. 

\vskip 3mm

Let $D$ be a sub-monoid of either $B$ or $B_{\overline{\Q}}$ and let $\mathbf{D}_L$ be the corresponding poset-category. As in the previous section, we know that points of $\widehat{\mathbf{D}_L}$ correspond to specific upwards closed subsets of $\mathbf{D}_L$. Repeating the argument here we get that the points correspond to equivalence classes of finite or infinite compositions $B_1 \circ B_2 \circ \hdots$, with all $B_i \in D$, under the relation
\[
B_1 \circ B_2 \circ \hdots \sim B'_1 \circ B'_2 \circ \hdots \Leftrightarrow \forall i, \exists j,k~:~\begin{cases}
B_1 \circ \hdots \circ B_i \geq B'_1 \circ \hdots \circ B'_j \\
B'_1 \circ \hdots \circ B'_i \geq B_1 \circ \hdots \circ B_k \end{cases} \]
Let $P_D = \mathbf{points}(\widehat{\mathbf{D}_L})$, then the {\em localic topology} of $P_D$ has as basis of open sets
\[
\mathbb{X}_l(B) = \{ p = [ B_1 \circ B_2 \circ \hdots ] \in P_D~|~\exists n~:~B_1 \circ \hdots \circ B_n \leq B \} \]
and as there are enough points to separate opens we have $\widehat{\mathbf{D}_L} \simeq \mathbf{Sh}(P_D,loc)$. The stalk of the sheaf $\mathcal{S}_{geo}$ or $\mathcal{S}_{arith}$ corresponding to the induced presheaf $S_{geo}$ or $S_{arith}$ on $\mathbf{D}_L$ at a point $p = [ B_1 \circ B_2 \circ \hdots ]$ is
\[
(\mathcal{S}_{geo})_p = \cup_n \overline{\Q}(B_1 \circ \hdots \circ B_n - t) \quad \text{and} \quad (\mathcal{S}_{arith})_p = \cup_n \Q(B_1 \circ \hdots \circ B_n - t) \]
If $D$ is a free sub-monoid on $n$ generators (see e.g. proposition~\ref{free}), then the points of $\widehat{\mathbf{D}_L}$ are in one-to-one correspondence with the finite or infinite paths from the root vertex in the complete rooted $n$-tree. The localic topology in this case has a basis of opens $U_{\gamma}$ consisting of all paths containing a fixed finite path $\gamma$.

\vskip 3mm

We will relate the above to the theory of {\em arboreal Galois representations}, see for example \cite{Jones}. For fixed $d$ let $T_n$ be the complete $d$-ary rooted tree of level $n$. Let $Aut(T_n)$ be the automorphism group of $T_n$ which is a subgroup of $S_{d^n}$ after choice of labelling of the leaf-vertices. Clearly, $Aut(T_1) \simeq S_d$ and by induction
\[
Aut(T_n) \simeq Aut(T_{n-1}) \wr Aut(T_1) \]
whence $Aut(T_n)$ is a subgroup of the $n$-fold iterated wreath product of $S_d$ by itself.

Let $P \in \Q[x]$ of degree $d$, and $\alpha \in \Q$ be such that for all $n$ we have that 
\[
P^{(n)}-\alpha = \underbrace{P \circ \hdots \circ P \circ P}_n - \alpha \]
has $d^n$ distinct roots in $\overline{\Q}$. Label the vertices of $T_n$ as follows: the root vertex will be $\alpha$, the vertices in the first layer will be the $d$ roots of $P-\alpha$, those in the second layer the roots of $P^{(2)}-\alpha$ grouped together so that their images under $P$ map to the same root of $P-\alpha$, and so on until the leaf-vertices which are the roots of $P^{(n)}-\alpha$ grouped together such that their images under $P$ map to the same root of $P^{(n-1)}-\alpha$. Let $K_n$ be the Galois extension of $\Q$ generated by the roots of $P^{(n)}-\alpha$, then the action of $Gal(K_n/\Q)$ gives a morphism
\[
\phi_n~:~Gal(K_n/\Q) \rInto Aut(T_n) \]
If $T_{\infty}$ is the infinite complete $d$-ary rooted tree, this procedure results in a continuous morphism
\[
\phi_P~:~Gal(\overline{\Q}/\Q) \rOnto G_{\infty,P,\alpha} = \underset{\infty \leftarrow n}{lim}~Gal(K_n/\Q) \rTo Aut(T_{\infty}) \]
which is called the {\em arboreal Galois representation} determined by $P$ and $\alpha$. 

By a result of Jones and Pink, \cite[Thm. 3.1]{Jones}, the index $[ T_{\infty} : G_{\infty,P,\alpha}]$ is infinite for post-critically finite polynomials $P$, hence also for all dynamical Belyi polynomials $B \in \Q[x]$. In \cite{Benedetto} the first explicit description of the arboreal representation of a cubic polynomial was given for $P=-2x^3+3x^2=B_{3,2}$ and this was extended in \cite{BouwDBM} to all rational Belyi polynomials $B_{d,k}$ described above.

\vskip 3mm

Let $D$ be the (free) sub-monoid generated by all dynamical Belyi polynomials $B_i \in \Q[x]$ of degree $d$ (among which are the $B_{d,k}$ for $0 \leq k < d$). Let $\alpha \in \Q \cap (0,1)$ be generic for the finite number of $B_i$, that is $B_i(\alpha) \not= 0,1$. Let $\mathcal{S}_{\alpha}$ be the sheaf on $(P_D,loc)$ corresponding to the contravariant functor
\[
S_{\alpha}~:~\mathbf{D}_L \rTo \mathbf{Sets} \qquad B \mapsto \Q(B-\alpha) \]

\begin{proposition} The stalk of the sheaf $\mathcal{S}_{\alpha}$ in the point $p=B_{i_1} \circ B_{i_2} \circ \hdots$ is an infinite Galois extension $K_p$ of $\Q$ giving rise to a generalised arboreal Galois representation of the absolute Galois group
\[
\phi_p~:~Gal(\overline{\Q}/\Q) \rOnto Gal(K_p/\Q) \rTo Aut(T_{\infty}) \]
where the $n$-th layer of $T_{\infty}$ is labeled by the $d^n$ distinct roots of $B_{i_1} \circ \hdots \circ B_{i_n}-\alpha$, grouped together in $d$-tuples having as their image under $B_{i_n}$ the same root of $B_{i_1} \circ \hdots \circ B_{i_{n-1}}-\alpha$. In particular, for $p=B_i \circ B_i \circ \hdots = B_i^{(\infty)}$ we recover the arboreal Galois representation corresponding to the dynamical Belyi polynomial $B_i$.
\end{proposition}

Because the Belyi monoids $B$ and $B_{\overline{\Q}}$ are right-cancellative we have a covariant functor $\pi : \mathbf{B}_{(\overline{\Q})L} \rTo \mathbf{B}_{(\overline{\Q})}$ which assigns to the unique morphism $B \rTo B'$ between two objects of $\mathbf{B}_{(\overline{\Q})L}$ such that $B \leq B'$ the unique element $B" \in B_{(\overline{\Q})}$ such that $B=B' \circ B"$. This functor induces a surjection on the points of the corresponding presheaf toposes giving us that
\[
\mathbf{points}(\widehat{B_{(\overline{\Q})}}) = [P_{B_{(\overline{\Q})}}] \]
where the equivalence relation on points is given by restricting to their tails.

The degree function $d$ of a dynamical Belyi polynomial gives a morphism of monoids $d : B \rTo \N^{\times}_+$ and $d : B_{\overline{\Q}} \rTo \N^{\times}_+$, and therefore geometric morphisms to the arithmetic site
\[
\widehat{\mathbf{B}} \rTo \widehat{\mathbf{A}} \quad \text{and} \quad \widehat{\mathbf{B}_{\overline{\Q}}} \rTo \widehat{\mathbf{A}} \]
Remains to link the Belyi-site(s) to the Conway-site. The approach below is inspired by \cite[Prop. 2.25]{ManinMarcolli}. Let $B$ and $B'$ be two dynamical Belyi polynomials, then we have seen that the number of black vertices of $D(B \circ B')$, that is,
\[
\# (B \circ B')^{-1}(0) = deg(B')(\# B^{-1}(0)-1) + \# (B')^{-1}(0). \]
But then, the map
\[
\beta~:~B_{(\overline{\Q})} \rTo C \qquad B \mapsto \begin{bmatrix} \frac{1}{deg(B)} & \frac{\# B^{-1}(0)-1}{deg(B)} \\ 0 & 1 \end{bmatrix} \]
is a morphism of monoids, with image the sub-monoid of $C$ generated by the $P_i$ with $0 \leq i < p$ for all prime numbers $p$ (look at the images of the $B_{d,k}$). Further, the hyperdistance $\delta(\beta(B),1) = deg(B)$, so the monoid map $\beta$ induces a geometric morphism between the sites giving rise to the commuting triangle
\[
\xymatrix{\widehat{\mathbf{B}} \ar@{->>}[rd]_{d} \ar@{->>}[rr]^{\beta} & & \widehat{\mathbf{C}} \ar@{->>}[ld]^{\delta} \\
& \widehat{\mathbf{A}} &} \]


\begin{thebibliography}{10}

\bibitem{Anderson}
Jacqueline Anderson, Irene Bouw, Ozlem Ejder, Neslihan Girgin, Valentijn Karemaker and Michelle Manes, {\em Dynamical Belyi maps}, {\tt arXiv:1703.08563} (2017)

\bibitem{Benedetto}
Robert L. Benedetto, Xander Faber, Benjamin Hutz, Jamie Juul and Yu Yasufuku, {\em A large arboreal Galois representation for a cubic postcritically finite polynomial}, {\tt arXiv:1612.03358} (2016)



\bibitem{BouwDBM}
Irene I. Bouw, \"Ozlem Ejder and Valentijn Karemaker, {\em Dynamical Belyi maps and arboreal Galois groups}, {\tt arXiv:1811.10086} (2018)

\bibitem{CCAS}
Alain Connes and Caterina Consani, {\it The Arithmetic Site, Le Site Arithm\'etique}, {\tt arXiv:1405.4527} (2014)

\bibitem{Conway}
John H. Conway, {\it Understanding groups like $\Gamma_0(N)$}, In "Groups, difference sets, and the Monster", Walter de Gruyter (1996)

\bibitem{HemelaerAT}
Jens Hemelaer, {\it Azumaya toposes}, {\tt arXiv:1707.03814} (2017)

\bibitem{Jones}
Rafe Jones, {\em Galois representations from pre-image trees: an arboreal survey}, Publ. Math. Besan\c{c}on Alg\`ebre Th\'eorie (2013)

\bibitem{Lando}
Sergei K. Lando and Alexander K. Zvonkin, {\em Graphs on surfaces and their applications}, Encyclopaedia of mathematical sciences, Vol. 141, Springer (2004)


\bibitem{LBSieve}
Lieven Le Bruyn, {\it The sieve topology on the arithmetic site}, Journal of Algebra and Applications (2016)

\bibitem{LBcovers}
Lieven Le Bruyn, {\it Covers of the arithmetic site}, {\tt arXiv:1602.01627v2} (2017)

\bibitem{LBMoonshine}
Lieven Le Bruyn, {\it The monstrous moonshine picture}, {\tt arXiv:1804.04127} (2018)

\bibitem{ManinMarcolli}
Yuri I. Manin and Matilde Marcolli, {\em Quantum statistical mechanics of the absolute Galois group}, {\tt arXiv:1907.13545v1} (2019)

\bibitem{MTabBC}
Matilde Marcolli and Goncalo Tabuada, {\it Bost-Connes systems, categorification, quantum statistical mechanics, and Weil numbers}, {\tt arXiv:1411.3223} (2014)

\bibitem{Shabat}
George Shabat and Alexander Zvonkin, {\em Plane trees and algebraic numbers}, AMS Contemporary Mathematics 178 (1994) 233-275


\bibitem{Wood}
Melanie Wood, {\em Belyi-extending maps and the Galois action on dessins d'enfant}, {\tt arXiv:0304489v2} (2005)

\end{thebibliography}
\end{document}